\documentclass[a4paper,11pt,twoside]{article}

\usepackage[latin1]{inputenc} 
\usepackage[T1]{fontenc} 
\newlength{\minipagewidth}
\usepackage{amsmath,amsthm,amsthm}
\usepackage{amssymb}
\usepackage{a4wide}
\usepackage{graphicx}
\usepackage{color}
\usepackage{epic}
\usepackage{color}
\usepackage{enumerate}

\usepackage[normalem]{ulem}

\usepackage{hyperref}

\usepackage{dsfont}

\newtheorem{theo}{Theorem}[section]
\newtheorem{algo}[theo]{Algorithm}

\newtheorem{rem}[theo]{Remark}

\newtheorem{propo}[theo]{Proposition}
\newtheorem{lemme}[theo]{Lemma}
\newtheorem{defi}[theo]{Definition}

\newtheorem{assumption}[theo]{Assumption}
\newcommand{\E}{\mathbb{E}}
\newcommand{\R}{\mathbb{R}}
\newcommand{\PP}{\mathbb{P}}

\newcommand{\N}{\mathbb{N}}

\usepackage{authblk}

\providecommand{\keywords}[1]{\textit{Keywords:} #1}
\providecommand{\subjclass}[1]{\textit{MSC:} #1}

\begin{document}

\title{Large deviations principle for the Adaptive Multilevel Splitting Algorithm in an idealized setting }

\author{Charles-Edouard Br\'ehier \footnote{Institut de Math\'ematiques, Universit\'e de Neuch\^atel, Rue Emile Argand 11, CH-2000 Neuch\^atel. e-mail: charles-edouard.brehier@unine.ch}}


\date{}

\maketitle

\begin{abstract}
The Adaptive Multilevel Splitting (AMS) algorithm is a powerful and versatile method for the simulation of rare events. It is based on an interacting (via a mutation-selection procedure) system of replicas, and depends on two integer parameters: $n\in \N^*$ the size of the system and the number $k\in\left\{1,\ldots,n-1\right\}$ of the replicas that are eliminated and resampled at each iteration.


In an idealized setting, we analyze the performance of this algorithm in terms of a Large Deviations Principle when $n$ goes to infinity, for the estimation of the (small) probability $\PP(X>a)$ where $a$ is a given threshold and $X$ is real-valued random variable. The proof uses the technique introduced in \cite{BLR}: in order to study the log-Laplace transform, we rely on an auxiliary functional equation.

Such Large Deviations Principle results are potentially useful to study the algorithm beyond the idealized setting, in particular to compute rare transitions probabilities for complex high-dimensional stochastic processes.

\end{abstract}



\section{Introduction}

In many problems from engineering, biology, chemistry, physics or finance, rare events are often critical and have a huge impact on the phenomena which are studied. From a general mathematical perspective, we may consider the following situation: let $(X_t)_{t\in \mathbb{T}}$, where $\mathbb{T}=\N$ or $\R$, be a (discrete or continuous in time) stochastic process, taking values in $\R^d$. Assume that $A,B\subset \R^d$ are two \emph{metastable regions}: starting from a neighborhood of $A$ (resp. of $B$), the probability that the process reaches $B$ (resp. $A$) before hitting $A$ (resp. $B$) is very small (typically, less than $10^{-10}$). As a consequence, a direct numerical Monte-Carlo with an ensemble of size $N$ does not provide significant results when $N$ is reasonably large (typically, less than $10^{10}$) in real-life applications.

Even if theoretical asymptotic expansions on quantities of interest are available - such as the Kramers-Arrhenius law given for instance by the Freidlin-Wentzell Large Deviations Theory or Potential Theory for the exit problem of a diffusion process in the small noise regime - in practice their explicit computation is not possible (for instance when the dimension is large) and numerical simulations are unavoidable.

It is thus essential to propose efficient and general methods, and to rigorously study their consistency and efficiency properties. Two main families of methods have been introduced in the 1950's and studied extensively since then, in order to improve the Monte-Carlo simulation algorithms, in particular for rare events: importance sampling and importance splitting (see for instance \cite{AG}, \cite{RubinoTuffin2009} for general reviews of these methods and \cite{KahnHarris1951} for the historical introduction of importance splitting). The main difference between these two methods is the following: the first one is \emph{intrusive}, meaning that the dynamics of the stochastic process (more generally, the distribution of the random variable of interest) is modified so that the probability that the event of interest increases and in a Monte-Carlo simulation it is realized more often, while the second is not intrusive and can thus be used more directly for complex problems. Instead, for importance splitting strategies, the state space is decomposed as a nested sequence of regions which are visited sequentially and more easily by an interacting system of replicas.

In this paper, we focus on an importance splitting strategy which is known as the Multilevel Splitting approach and describe it in the following setting. Let $h:\R^d\rightarrow \R$ be a given function and assume we want to estimate the probability $p=\PP(X> a)$ that a real-valued random variable $X=h(Y)$ (where $Y$ is a $\R^d$-valued random variable) belongs to $(a,+\infty)$ for a given threshold $a\in \R$. This situation is not restrictive for many applications; indeed, we may take $X=\mathds{1}_{\tau_B<\tau_A}$ and any $a\in(0,1)$ in the situation described above, where $\tau_A$ and $\tau_B$ are the hitting times of $A$ and $B$ by the process $X$. A key assumption on the distribution of $X$ is the following: we assume that the cumulative distribution function $F$ of $X$ - {\it i.e.} $F(x)=\PP(X\leq x)$ for any $x\in \R$ - is continuous; for convenience, we also assume that $F(0)=0$ - {\it i.e.} $X>0$ almost surely.

The multilevel splitting approach (see \cite{KahnHarris1951}, \cite{GlassermanHeidelbergerShahabuddinZajic1999}, \cite{CerouDel-MoralFuronGuyader2012} for instance) is based on the following decomposition of $p$ as a telescopic product of conditional probabilities:
\begin{equation}\label{eq:decomp_p_intro}
p=\PP(X>a)=\prod_{i=1}^{N}\PP(X>a_i \big| X>a_{i-1}),
\end{equation}
where $a_0=0<a_1<\ldots<a_N=a$ is a sequence of non-decreasing {\it i.e. } levels. In other words, the realization of the event $\left\{X>a\right\}$ is split into the realizations of the $N$ events $\left\{X>a_i\right\}$ conditional on $\left\{X>a_{i-1}\right\}$; each event has a larger probability than the initial one and is thus much easier to realize. Then each of the conditional probabilities is estimated separately, for instance with independent Monte-Carlo simulations, or using a Sequential Monte-Carlo technique with a splitting of successful trajectories. This approach have been studied with different viewpoints and variants under different names in the literature -  nested sampling \cite{Skilling2006}, \cite{Skilling2007}, subset simulation \cite{AB}, RESTART, \cite{Villen-Altamirano1991}, \cite{Villen-Altamirano1994}. 

For future reference, we introduce the following (unbiased) estimator of $p$ given by the multilevel splitting approach with $N$ levels and $n$ replicas:
\begin{equation}\label{eq:estim_MS_intro}
\hat{p}_n^N=\prod_{i=1}^{N}\frac{1}{n}\sum_{m=1}^{n}\mathds{1}_{X_{m}^{(i)}>a_{i}},
\end{equation}
where the random variables $(X_{m}^{(i)})_{1\leq m\leq N, 1\leq i\leq N}$ are independent and the distribution of $X_{m}^{(i)}$ is $\mathcal{L}(X | X>a_{i-1})$. Thus $\hat{p}_n^N$ is a product of $N$ independent Monte-Carlo estimators of the conditional probabilities in \eqref{eq:decomp_p_intro}.

The efficiency of the algorithm depends crucially on the choice of the sequence of levels $(a_i)_{1\leq i\leq N}$: for a fixed size $N$, the variance of the estimator is minimized when the conditional probabilities are equal (to $p^{1/N}$); moreover the associated variance converges (to $-p^2\log(p)/n$) when $N$ goes to infinity - see for instance \cite{CerouDel-MoralFuronGuyader2012} for more details.

To get a more flexible algorithms, a possible approach is to compute levels adaptively, as proposed in \cite{CR}, and studied extensively in the last years, see for instance \cite{BLR}, \cite{BGT}, \cite{CerouGuyader2014}, \cite{G..}, \cite{Simonnet2014}, \cite{Walter2014}. It is essential to check that these adaptive versions still give reliable results, and to prove they do it efficiently.

More precisely, we consider the Algorithm \ref{algo:AMS} defined below, which depends on two parameters $n$ and $k$, with the condition $1\leq k\leq n-1$. We let evolve a system of $n$ interacting replicas, and at each iteration a selection-mutation procedure leads to resample the system as follows: we compute the $k$-th order statistic $Z$ - which corresponds to the so-called level at the given iteration - of the system and eliminate the $k$ replicas with values less than $Z$; they are then resampled using the conditional distribution $\mathcal{L}(X | X>Z)$ of $X$ conditional on $\left\{X>Z\right\}$. The algorithm stops when $Z\geq a$, and we define an estimator $\hat{p}^{n,k}$ depending on the number of iterations and of the terminal configuration of the system of replicas, see \eqref{eq:estimator}. In practice, we require to be able to sample according to the conditional distribution $\mathcal{L}(X | X>z)$ for any value of $z$: this is part of the idealized setting assumption; even if it is rarely satisfied in real-life applications, the study of the algorithm in that setting is already challenging and yields very interesting results, that can usually be generalized beyond this simplified case at the price of a much more intricate analysis.

Let us recall a few fundamental results. In \cite{G..} (see also \cite{Simonnet2014}, \cite{Walter2014}), it was proved that for any value of $n\geq 2$ then $\hat{p}^{n,1}$ is an unbiased estimator of $p$ - meaning that $\E[\hat{p}^{n,1}]=p$. This result was extend to general $1\leq k\leq n-1$ in \cite{BLR}. Efficiency properties have been studied with the proof of Central Limit Theorems in two different kinds of regimes: either $k$ is fixed and $n\rightarrow +\infty$ (see \cite{BGT} as well as \cite{G..} and \cite{Simonnet2014} when $k=1$), or both $k$ and $n$ go to infinity, in such a way that $k/n$ converges to $\alpha\in(0,1)$ - which gives a fixed proportion of resampled replicas at each iteration, see \cite{CR} and the more recent work \cite{CerouGuyader2014} in a very general framework.

The efficiency is ensured by the observation that the asymptotic variance is the same for both the adaptive and the non-adaptive versions. Moreover, it is much smaller than when using a crude Monte-Carlo estimator, {\it i.e.} the empirical average
\begin{equation}\label{eq:cMC_intro}
\overline{p}_n=\frac{1}{n}\sum_{m=1}^{n}\mathds{1}_{X_m>a},
\end{equation}
where the random variables $(X_m)_{1\leq m\leq n}$ are independent and identically distributed, with distribution $\mathcal{L}(X)$.

In this paper, we prove a similar result with a different criterion, which seems to be original compared with existing literature: we prove a Large Deviations Principle principle for the distribution of the estimator $\hat{p}^{n,k}$ given by the adaptive algorithm when $k$ is fixed and $n\rightarrow +\infty$. Our main result is Theorem \ref{th:LDP}, which in particular yields for any given $\epsilon>0$
$$
\frac{1}{n}\log\Bigl(\PP\bigl( |\hat{p}^{n,k}-p| \ge \epsilon \bigr)\Bigr)\underset{n\rightarrow +\infty}\rightarrow -\min\bigl(I(p+\epsilon),I(p-\epsilon)\bigr)<0.
$$
The rate function $I$ - see \eqref{eq:rate_I} - obtained in Theorem \ref{th:LDP} does not depend on $k$. We then compare this rate function with $\mathcal{I}$ - see \eqref{eq:def_Ical} - the rate function obtained for a crude Monte-Carlo estimator $\overline{p}_n$ given by \eqref{eq:cMC_intro} (thanks to Cramer Theorem, see \cite{DZ}) and show that for any $y\in (0,1)\setminus p$ we have $I(y)>\mathcal{I}(y)$ - we have $\mathcal{I}(p)=I(p)=0$, and $\mathcal{I}(y)=I(y)=+\infty$ if $y\notin (0,1)$ - and thus
$$\frac{\PP(\hat{p}^{n,k}-p>\epsilon)}{\PP(\overline{p}_n-p>\epsilon)}\underset{n\rightarrow +\infty}\rightarrow 0.$$
In other words, for large $n$, the probability that $\hat{p}^{n,k}$ deviates from $p$ from above (and similarly from below) with threshold $\epsilon>0$ decreases exponentially fast, at a faster rate than for $\overline{p}_n$.

Moreover, we prove that the non-adaptive, fixed-levels estimator $\hat{p}_n^N$ satisfies a Large Deviations Principle when $n\rightarrow +\infty$ with rate function $\mathcal{I}_N$ for a fixed number of levels $N$ and when the levels are chosen in an optimal way, namely such that $\PP(X>a_{i} | X>a_{i-1})=p^{1/N}$ does not depend on $i$. We then show that $\lim_{N\rightarrow +\infty}\mathcal{I}_N(y)\leq I(y)$ for any $y\in\R$: this inequality is sufficient to prove that asymptotically the adaptive algorithm performs (at least) as well as the non-adaptive version in this setting, in terms of Large Deviations.

The proof of Theorem \ref{th:LDP} relies on the technique introduced in \cite{BLR}. First, we restrict the study of the properties of the algorithm to the case when $X$ is exponentially distributed with parameter $1$ (this key remark was introduced first in \cite{G..} and used also in \cite{Simonnet2014}, \cite{Walter2014}). Instead of working on $\hat{p}^{n,k}$ directly, we focus on its logarithm $\log(\hat{p}^{n,k})$, and prove that when considering the algorithm as depending on an initial condition $x$, the Laplace transform of the latter is solution of a functional (integral) equation (with respect to the $x$ variable) - thanks to a decomposition of the realizations of the algorithm according to the value of the first level. To study the equation in the asymptotic regime considered in this paper, we then derive a linear ordinary differential equation of order $k$ and perform an asymptotic expansion. Note that we do not give all details for the derivation of the differential equations and the basic properties of its coefficients; for some points we refer the reader to \cite{BLR} where all the arguments are proved with details and here we mainly focus on the proof of the new asymptotic results as well as on the interpretation of the Large Deviations Principle for our purpose.

It seems that studying the performance of multilevel splitting algorithms via Large Deviations Principle is an original approach, which can complement the more classical studies which are all based on Central Limit Theorems. In this paper, we proved a result in a specific regime ($k$ is fixed, $n\rightarrow +\infty$) in the idealized setting. To go further, it would be interesting to look at other regimes ($k,n\rightarrow +\infty$ with $k/n\rightarrow \alpha\in(0,1)$) and to go beyond the idealized setting. This will be the subject of future investigation.

The paper is organized as follows. In Section \ref{sect:desc_AMS}, we introduce our main assumptions (Section \ref{sect:ass}), describe the Adaptive Multilevel Splitting algorithm (Section \ref{sect:algo}) and recall several of its fundamental properties used in the sequel of the article (Section \ref{sect:prop_AMS}). The main result of this paper is given in Section \ref{sect:LDP}: it is the Large Deviations Principle for the estimator of the probability given by the AMS estimator, see Theorem \ref{th:LDP}. An important auxiliary result is stated in Section \ref{sect:strategy}, and proofs are carried over in Section \ref{sect:proof} - some technical estimates being proved in Section \ref{sect:detailed_proof}. We compare the performance in terms of the Large Deviations Principle of the AMS algorithm with two other methods in Section \ref{sect:comp}: a crude Monte-Carlo method and a fixed-level splitting method. Finally, we give some concluding remarks and perspectives in Section \ref{sect:conclusion}.

\section{Description of the Adaptive Multilevel Splitting algorithm}\label{sect:desc_AMS}

\subsection{Assumptions}\label{sect:ass}

Let $X$ be some real random variable. For simplicity, we assume that $X>0$ almost surely.

We want to estimate the probability $p=\PP(X>a)$, where $a>0$ is some threshold. When $a$ goes to $+\infty$, $p$ goes to $0$ and we have to estimate the probability of a rare event.

We make a fundamental assumption on the distribution of $X$.
\begin{assumption}\label{ass:cdf:c0}
Let $F$ denote the cumulative distribution function of $X$: we assume that $F$ is continuous.
\end{assumption}

More generally, for both theoretical and practical purpose, we introduce for $0\leq x\leq a$ the conditional probability
\begin{equation}\label{eq:cond_P(x)}
P(x)=\PP(X>a |X>x);
\end{equation}
we also denote by $\mathcal{L}\bigl(X | X>x)$ the associated conditional distribution, and $F(\cdot; x)$ its cumulative distribution function: for any $y>x$ we have $F(y;x)=\frac{F(y)-F(x)}{1-F(x)}$ whenever $F(x)<1$.

We notice two important equalities: $P(a)=1$, and the estimated probability is $p=P(0)$; in fact, the distribution of $X$ is equal to $\mathcal{L}\bigl(X | X>0)$.

The {\emph idealized setting} refers to the following assumptions:
\begin{itemize}
\item Assumption \ref{ass:cdf:c0} is satisfied (\emph{theoretical condition});
\item it is possible to sample according to the conditional distribution $\mathcal{L}\bigl(X | X>x)$ for any $x\in[0,a)$ (\emph{practical condition}).
\end{itemize}

In view of a practical implementation of the algorithm, the second condition is probably the most restrictive. One may rely on some approximation of the conditional distribution $\mathcal{L}\bigl(X | X>x)$ thanks to a Metropolis-Hastings algorithm: in that case (see \cite{CerouGuyader2014} for instance), the analysis we develop here does not apply, but gives an interesting insight for the behavior in the case of a large number of steps in the Metropolis-Hastings auxiliary scheme (rigorously, we treat the case of an infinite number of steps).

\subsection{The algorithm}\label{sect:algo}

We now present the Adaptive Multilevel Splitting algorithm, under the assumptions of Section \ref{sect:ass} above.

The algorithm depends on two parameters:
\begin{itemize}
\item the number of replicas $n$;
\item the number $k\in\left\{1,\ldots,n-1\right\}$ of replicas that are resampled at each iteration.
\end{itemize}

The other necessary parameters are the initial condition $x$ and the stopping threshold $a$: the aim is to estimate the conditional probability $P(x)$ introduced in \eqref{eq:cond_P(x)}. For future reference, we denote by ${\rm AMS}(n,k;a,x)$ the algorithm.

The dependence with respect to $x$ allows us below to state fundamental functional equations on useful observables of the estimator computed at the end of the iterations of the algorithm, as a function of $x$. In practice, we are interested in the case $x=0$; in this situation, the algorithm is denoted by ${\rm AMS}(n,k;a)$.

Before we detail the algorithm, we introduce important notation. First, when we consider a random variable $X_{i}^{j}$, the subscript $i$ denotes the index in $\left\{1,\ldots,n\right\}$ of a replica, while the superscript $j$ denotes the iteration of the algorithm.

Moreover, we use the following notation for order statistics. Let $Y=(Y_1,\ldots,Y_n)$ be independent and identically distributed (i.i.d.) real valued random variables with continuous cumulative distribution function; then there exists almost surely a unique (random) permutation $\sigma$ of $\left\{1,\ldots,n\right\}$ such that $Y_{\sigma(1)}<\ldots<Y_{\sigma(n)}$. For any $k \in \{1, \ldots,n\}$, we then denote by $Y_{(k)}=Y_{\sigma(k)}$ the so-called $k$-th order statistic of the sample $Y$. Sometimes we need to specify the size of the sample of which we consider the order statistics: we then use the notation $Y_{(k,n)}$.

We are now in position to write the ${\rm AMS}(n,k;a,x)$ algorithm.
\begin{algo}[Adaptive Multilevel Splitting, ${\rm AMS}(n,k;a,x)$]\label{algo:AMS}
~

\noindent
{\bf Initialization:}
Set the initial level $Z^{0}=x$.

Sample $n$ i.i.d. realizations $X_{1}^{0},\ldots,X_{n}^{0}$, with distribution $\mathcal{L}(X | X>x)$.

Define $Z^{1}=X_{(k)}^{0}$, the $k$-th order statistics of the sample $X^{0}=(X_{1}^{0},\ldots,X_{n}^{0})$, and $\sigma^1$ the (a.s.) unique associated permutation: $X_{\sigma^1(1)}^{0}<\ldots<X_{\sigma^1(n)}^{0}$. 

Set $j=1$.

\noindent
{\bf Iterations (on $j\geq 1$):} While $Z^{j} <  a$:

\begin{itemize}
\item Conditional on $Z^{j}$, sample $k$ new independent random variables $(Y_1^j,\ldots,Y_k^j)$, according to the law $\mathcal{L}(X | X>Z^{j})$.

\item Set
$$
X_{i}^{j}=\begin{cases}Y_{(\sigma^j)^{-1}(i)}^{j} \quad \text{if } (\sigma^j)^{-1}(i)\leq k\\ X_{i}^{j-1} \quad \text{if } (\sigma^j)^{-1}(i)>k. \end{cases}
$$

In other words, we resample exactly $k$ out of the $n$ replicas, namely those with index $i$ such that $X_{i}^{j-1}\leq Z^{j}$, \textit{i.e.} such that $i\in\left\{\sigma^{j}(1),\ldots,\sigma^{j}(k)\right\}$ (which is equivalent to $(\sigma^j)^{-1}(i)\leq k$). They are resampled according the the conditional distribution $\mathcal{L}(X | X>Z^{j})$. The other replicas are not modified.


\item Define $Z^{j+1}=X_{(k)}^{j}$, the $k$-th order statistics of the sample $X^{j}=(X_{1}^{j},\ldots,X_{n}^{j})$, and $\sigma^{j+1}$ the (a.s.) unique associated permutation: $X_{\sigma^{j+1}(1)}^{j}<\ldots<X_{\sigma^{j+1}(n)}^{j}$. 

\item Finally increment $j\leftarrow j+1$.

\end{itemize}

\noindent
{\bf End of the algorithm:}
Define $J^{n,k}(x)=j-1$ as the (random) number of iterations. Notice that $J^{n,k}(x)$ is such that $Z^{J^{n,k}(x)} < a$ and $Z^{J^{n,k}(x)+1} \ge a$.
\end{algo}

Notice for instance that $J^{n,k}(x)=0$ if and only if $Z^{1}>a$: we mean that in this case the algorithm has required $0$ iteration, since the stopping condition at the beginning of the loop (on $j$) is satisfied without entering into the loop.

The estimator of the probability $P(x)$ is defined by
\begin{equation}\label{eq:estimator}
\hat{p}^{n,k}(x)=C^{n,k}(x)\left(1-\frac{k}{n}\right)^{J^{n,k}(x)},
\end{equation}
with
\begin{equation}\label{eq:corrector}
C^{n,k}(x)=\frac{1}{n}{\rm Card}\left\{i ;\, X_{i}^{J^{n,k}(x)} \ge a\right\}.
\end{equation}

The interpretation of the factor $C^{n,k}(x)$ is the following: it is the proportion of the replicas $X_{i}^{J^{n,k}(x)}$ which satisfy $X_{i}^{j}\ge a$: since $X_{(k)}^{J^{n,k}(x)}=Z^{J^{n,k}(x)+1}\ge a$, we have $C^{n,k}(x)\ge \frac{n-k+1}{n}$. Notice that $C^{n,1}(x)=1$.


When $x=0$, to simplify notations we set $\hat{p}^{n,k}=\hat{p}^{n,k}(0)$.

\subsection{Properties of the AMS Algorithm \ref{algo:AMS}}\label{sect:prop_AMS}

\subsubsection*{Well-posedness}

We first recall some important results on the well-posedness of the algorithm. For more detailed statements and complete proofs, see Section $3.2$ in \cite{BLR}, in particular Proposition $3.2$ there.

First, at each iteration $j$ of the algorithm, conditional on the level $Z^{j}$, the resampling produces a family of $n$ random variables $\bigl(X_{i}^{j}\bigr)_{1\leq i\leq n}$ which are independent and identically distributed, with distribution $\mathcal{L}(X | X>Z^{j})$. By Assumption \ref{ass:cdf:c0}, conditional on $Z^{j}$ the latter conditional distribution also admits a continuous cumulative distribution function $F(\cdot; Z^{j})$; as a consequence, almost surely the permutation $\sigma^{j+1}$ is unique, and the level $Z^{j+1}$ is well-defined.

Moreover, if we assume that $P(x)>0$, almost surely the algorithm stops after a finite number of steps, for any values of $k$ and $n$ such that $1\leq k\leq n-1$: the random variable $J^{n,k}(x)$ almost surely takes values in $\N$, and the estimator $\hat{p}^{n,k}(x)$ is well-defined and takes values in $(0,1]$.

\subsubsection*{Reduction to the exponential case}

We now state properties that are essential for our theoretical study of the algorithm below.

One of the main tools in \cite{BLR} and \cite{BGT}, which was also used in \cite{G..} in the case $k=1$, is the restriction to the case where the random variables are exponentially distributed. More precisely, assume that $P(x)>0$, and denote by $\mathcal{E}(1)$ the exponential distribution with mean $1$. Then in distribution the algorithm ${\rm AMS}(n,k;a)$ is equal to the algorithm ${\rm AMS}_{\rm expo}(n,k;-\log(p))$ in which we assume that the distribution is $\mathcal{E}(1)$; a similar result holds for ${\rm AMS}(n,k;a,x)$ when $x\in [0,a)$. In particular, the associated estimators are equally distributed. The main argument is the well-known equality of distribution $F(X)=U$ where $U$ is uniformly distributed on $(0,1)$.

In the sequel, we state in Section \ref{sect:LDP} our results in the general setting - {\it i.e.} for ${\rm AMS}(n,k;a)$, with the probability $p$ and the estimator $\hat{p}^{n,k}$ - but in the remaining of the paper we give proofs in the exponential case, namely for ${\rm AMS}_{\rm expo}(n,k;a_{\rm expo},x)$ with $a_{\rm expo}=-\log(p)$, and we omit the reference to the exponential case to simplify the notation. Whether we consider the general or the exponential case will be clear from the context.

\section{The Large Deviations Principle result for the AMS algorithm}\label{sect:LDP}

The main result of this article is the following Theorem \ref{th:LDP}, which states a Large Deviations Principle (in the sense of \cite{DZ}) for the distribution $\mu^{n,k}=\mathcal{L}\bigl(\hat{p}^{n,k}\bigr)$ of $\hat{p}^{n,k}$ for fixed probability $p>0$ and $k\in\N^*$, in the limit $n\rightarrow +\infty$.

\begin{theo}\label{th:LDP}
Assume that $p\in(0,1)$ and $k\in\N^*$ are fixed. Then the sequence $\bigl(\mu^{n,k}\bigr)_{n\in\N, n>k}$ of distributions of the estimator $\hat{p}^{n,k}$ of $p$ obtained by the ${\rm AMS}(n,k;a)$ algorithm satisfies a Large Deviations Principle with the rate function $I$ defined by
\begin{equation}\label{eq:rate_I}
I(y)=\begin{cases} +\infty \text{ if } y\notin (0,1)\\ \log(y)\log(\frac{\log(p)}{\log(y)})+\log(\frac{y}{p}) \text{ if } y\in (0,1). \end{cases}
\end{equation}
\end{theo}

We observe that the rate function does not depend on $k$.

Notice that the statement above is restricted to $p\in(0,1)$. Indeed, when $p=1$, we have almost surely $\hat{p}^{n,k}=1$ (the algorithm stops after $0$ iteration). Moreover, we always estimate the probability of events which have a positive probability (otherwise the algorithm does not stop after a finite number of iterations).

The following Proposition describes some properties of the rate function $I$.
\begin{propo}
The rate function $I$ is of class $\mathcal{C}^{\infty}$ on its domain $(0,1)$.

Moreover, $p$ is the unique minimizer of $I$: we have $I(p)=I'(p)=0$, $I''(p)=\frac{1}{-p^2\log(p)}>0$.

Finally, for any $y\in(0,1)\setminus\left\{p\right\}$ we have $I(y)>0$; $I$ is decreasing on $(0,p)$ and is increasing on $(p,1)$.
\end{propo}

\begin{proof}
Straightforward computations yield that for $y\in(0,1)$ we have
\begin{gather*}
\frac{d I(y)}{dy}=\frac{\log(\log(p))-\log(\log(y))}{y},\\
\frac{d^2 I(y)}{dy^2}=-\frac{\log(\log(p))-\log(\log(y))}{y^2}-\frac{1}{y^2\log(y)}.
\end{gather*}
\end{proof}

Let $\epsilon\in(0,\max(p,1-p))$; then from Theorem \ref{th:LDP} we have when $n\rightarrow +\infty$
\begin{equation}\label{eq:appli_LDP_epsilon}
\frac{1}{n}\log\Bigl(\PP\bigl( |\hat{p}^{n,k}-p| \ge \epsilon \bigr)\Bigr)\underset{n\rightarrow +\infty}\rightarrow -\min\bigl(I(p+\epsilon),I(p-\epsilon)\bigr)<0.
\end{equation}
Applying the Borel-Cantelli Lemma, we get the almost sure convergence $\hat{p}^{n,k}\rightarrow p$.

\begin{rem}
The almost sure limit is consistent with the unbiasedness result ($\E[\hat{p}^{n,k}]=p$) from \cite{BLR}. There we were only able to prove the convergence in probability of $\hat{p}^{n,k}$ to $p$.

Notice also that in \cite{BGT} we proved a Central Limit Theorem: $$\sqrt{n}\bigl(\hat{p}^{n,k}-p\bigr)\rightarrow \mathcal{N}\bigl(0,-p^2\log(p)\bigr).$$ The  asymptotic variance is given by $I''(p)$.
\end{rem}

We conclude this section with a result showing that the choice of the regime $p$ (and $k$) fixed and $n\rightarrow +\infty$ is crucial to get Theorem \ref{th:LDP}. Indeed, set $k=1$, and for a given $\sigma>0$ assume that $n$ and $p$ are related though the following formula: $-\log(p)=\sigma^2 n$. Then $\frac{p^{n,k}}{p}$ converges (in law) to a log-normal distribution, as stated in the following proposition.

\begin{propo}\label{propo:log-norm}
If $-\log(p)=\sigma^2 n$, we have the convergence in distribution
$$\lim_{n \rightarrow \infty} \frac{\hat{p}^{n,1}}{p} =\exp(\sigma Z-\sigma^2/2),$$
where $Z\sim \mathcal{N}(0,1)$.
\end{propo}

The proof is postponed to Section \ref{sect:proof_1}, since it uses the same arguments as the proof of Theorem \ref{th:LDP} in the case $k=1$.

Let $\epsilon>0$. Then (compare with \eqref{eq:appli_LDP_epsilon} with $\epsilon p$ instead of $\epsilon$)
$$\PP\bigl( |\frac{\hat{p}^{n,1}}{p}-1| \ge \epsilon \bigr)\underset{n=-\frac{\log(p)}{\sigma^2}\rightarrow +\infty}\longrightarrow \PP_{Z\sim \mathcal{N}(0,1)}\bigl( |\exp(\sigma Z-\sigma^2/2)-1| \ge \epsilon \bigr)\Bigr)>0,$$
where the limit is positive, while owing to \eqref{eq:appli_LDP_epsilon} when $p$ fixed, $\PP\bigl( |\frac{\hat{p}^{n,1}}{p}-1| \ge \epsilon \bigr)$ converges to $0$ exponentially fast when $n\rightarrow +\infty$.


\section{Strategy of the proof}\label{sect:strategy}

To prove Theorem \ref{th:LDP}, we in fact first prove a Large Deviations Principle for $\tilde{\mu}^{n,k}=\mathcal{L}\bigl( \log(\hat{p}^{n,k})\bigr)$, with rate function $J$ given below.

\begin{propo}\label{propo:LDP_log}
Assume that $p\in(0,1)$ and $k\in\N^*$ are fixed. Then the sequence $\bigl(\tilde{\mu}^{n,k}\bigr)_{n\in \N, n>k}$ of distributions of $\log(\hat{p}^{n,k})$ obtained by the ${\rm AMS}(n,k;a)$ algorithm satisfies a Large Deviations Principle with the rate function $J$ defined by
\begin{equation}\label{eq:rate_J}
J(z)=\begin{cases} +\infty \text{ if } z\geq 0\\ z-\log(p)-z\log(\frac{z}{\log(p)}) \text{ if } z<0. \end{cases}
\end{equation}
\end{propo}

Then Theorem \ref{th:LDP} immediately follows from Proposition \ref{propo:LDP_log} and the application of the contraction principle (see \cite{DZ}, Theorem $4.2.1$): we have $\hat{p}^{n,k}=\exp\bigl(\log(\hat{p}^{n,k})\bigr)$, and we obtain the rate function with the identity $I(y)=J(\log(y))$.

The proof of Proposition \ref{propo:LDP_log} relies on the use of the G\"artner-Ellis Theorem (see Theorem $2.3.6$ in \cite{DZ}) and the asymptotic analysis  when $n\rightarrow +\infty$ of the log-Laplace transform of $\tilde{\mu}^{n,k}$.

\begin{propo}\label{propo:Lambda}
Set for any $1\leq k\leq n-1$ and any $\lambda\in \R$
\begin{equation}\label{eq:Lambda}
\Lambda_{n,k}(\lambda)=\log \Bigl( \E \Bigl[\exp \bigr(\lambda \log(\hat{p}^{n,k}) \bigl) \Bigr] \Bigr).
\end{equation}
Then for any fixed $k\in\N^*$ and any $\lambda\in \R$ we have the convergence
\begin{equation}\label{eq:Lambda_conv}
\frac{1}{n}\Lambda_{n,k}(n\lambda)\rightarrow \Lambda(\lambda)=-\log(p)(\exp(-\lambda)-1).
\end{equation}
The Fenchel-Legendre transform $\Lambda^*$ of $\Lambda$ satisfies:
\begin{equation}\label{eq:Lambda^*}
\begin{aligned}
\Lambda^*(z)&=\sup_{\lambda\in \R}\bigl(\lambda z-\Lambda(\lambda)\bigr)\\
&=\begin{cases}+\infty \text{ if } z\geq 0\\ z-\log(p)-z\log(\frac{z}{\log(p)})   \text{ if } z<0. \end{cases}
\end{aligned}
\end{equation}
\end{propo}

Then for any $k\in \N^*$, the sequence of distributions $\bigl(\tilde{\mu}^{n,k}\bigr)_{n\in\N, n>k}$ satisfies a Large Deviations Principle, with the rate function $J=\Lambda^*$.

The proof of \eqref{eq:Lambda_conv} is the main task of this paper. In Section \ref{sect:proof_1}, we give a first easy proof in the case $k=1$, relying on the knowledge of the distribution of $J^{n,1}$: it is a Poisson distribution with mean $-n\log(p)$. We can then compute explicitly $\Lambda_{n,1}(\lambda)$ and prove \eqref{eq:Lambda_conv}. In Section \ref{sect:proof_k}, we study the general case $k\geq 1$ with the method introduced in \cite{BLR}, in the exponential case: for the algorithm ${\rm AMS}_{\rm expo}(n,k;a,x)$, we derive a functional equation on the Laplace transform $\exp\bigl(\Lambda_{n,k}(\lambda)$ as a function of the initial condition $x$, for fixed parameter $\lambda$.

For completeness, we close this Section with the computation of the Fenchel-Legendre transform $J=\Lambda^*$ of $\Lambda$ in Proposition \ref{propo:Lambda}.

\begin{proof}
First, assume that $z\geq 0$. Then $\lambda z-\Lambda(\lambda)\rightarrow +\infty$ when $\lambda\rightarrow +\infty$: thus $\Lambda^*(z)=+\infty$. Notice that this result is not surprising, since $\log(\hat{p}^{n,k})<0$ almost surely.

If $z<0$, the map $\lambda\in \R\mapsto \lambda z-\Lambda(\lambda)$ admits the limit $-\infty$ for $z\rightarrow \pm \infty$, and attains its maximum at the unique solution $\lambda_z$ of the equation $z-\frac{d\Lambda(\lambda)}{d\lambda}(\lambda_z)=0$, which is given by $\lambda_z=-\log\bigl(\frac{z}{\log(p)}\bigr)$. Then $\Lambda^*(z)=\lambda_z z-\Lambda(\lambda_z)$, which gives \eqref{eq:Lambda^*}.
\end{proof}

\section{Proof of Proposition \ref{propo:Lambda}}\label{sect:proof}

\subsection{The case $k=1$}\label{sect:proof_1}

We start with a proof of Theorem \ref{th:LDP} when $k=1$: in this case, we have $C^{n,1}=1$ almost surely, and the number of iterations $J^{n,1}$ follows a Poisson distribution $\mathcal{P}(-n\log(p))$ (see for instance \cite{BLR}, \cite{G..}).

As a consequence, it is very easy to prove Proposition \ref{propo:LDP_log}. Let $\lambda\in\R$. Then
\begin{align*}
\Gamma_{n,1}(\lambda)&=\exp\big(\Lambda_{n,1}(\lambda)\bigr)\\
&=\E\big[\exp\bigl(\lambda\log(\hat{p}^{n,1})\bigr)\big]\\
&=\E\big[\exp\bigl(\lambda\log(1-1/n) J^{n,1}\bigr)\big]\\
&=\exp\Bigl(-n\log(p)\bigl(\exp(\lambda\log(1-1/n))-1\bigr)\Bigr).
\end{align*}

It is now easy to conclude: when $n\rightarrow +\infty$
\begin{align*}
\frac{1}{n}\log\big(\Lambda(n\lambda)\big)&=-\log(p)\big(\exp(n\lambda\log(1-1/n))-1\big)\\
&\underset{n\rightarrow+\infty}\rightarrow -\log(p)\big(\exp(-\lambda)-1\big).
\end{align*}

We have performed explicit calculations, using the knowledge of the distribution of $J^{n,1}$. However for $k>1$, we cannot rely on such simple arguments and we need other tools.

We would like to use the connexion with the Poisson distribution in order to give an interpretation of the rate functions $I$ and $J$. More precisely, $I$ is obtained from $J$ by the contraction principle ($I(y)=J(\log(y))$), and $J$ is the rate function obtained in the Cramer theorem where the distribution $R$ is such that $-R\sim\mathcal{P}\bigl(-\log(p)\bigr)$. Indeed, let $(R_m)_{m\in \N^*}$ be independent, with the same distribution as $X$; if we denote by $\overline{R}_n=\frac{1}{n}\sum_{m=1}^{n}R_m$ the empirical average, we compute for any $\lambda\in \R$
\begin{align*}
\E\big[\exp\bigl(n\lambda\overline{R}_n\bigr)\big]&=\Bigl(\E\big[\exp\bigl(\lambda R\bigr) \Bigr)^{n}\\
&=\Bigl(\exp\bigl(-\log(p)\bigl(\exp(-\lambda)-1\bigr)\bigr) \Bigr)^{n}.
\end{align*}

To conclude this section on the case $k=1$, we prove Proposition \ref{propo:log-norm}. We use again the explicit knowledge of the distribution of $J^{n,1}$ and use a Central Limit Theorem on exponential distributions to conclude.
\begin{proof}[Proof of Proposition \ref{propo:log-norm}]
We write (with $a=-\log(p)=\sigma^2 n$)
\begin{align*}
\frac{\hat{p}^{n,1}}{p}&=\exp(J^{n,1}\log(1-1/n)+a)\\
&=\exp\left(\frac{J^{n,1}-n a}{\sqrt{n a}}\sqrt{n a}\log(1-1/n)+a+n a\ln(1-1/n)\right).
\end{align*}
By the Central Limit Theorem on the Poisson distribution, one gets, in the limit $n \rightarrow +\infty$, the following convergence in distribution
$$\frac{J^{n,1}-n a}{\sqrt{n a}}\rightarrow \mathcal{N}(0,1).$$
Moreover, when $n\rightarrow +\infty$, we have $\sqrt{n a}\log(1-1/n)=n\sigma \log(1-1/n)\rightarrow -\sigma$ and $a+n a\log(1-1/n)=\sigma^2\left(n+n^2\ln(1-1/n)\right)$ tends to $\frac{-\sigma^2}{2}$. This concludes the proof.

\end{proof}

\subsection{The general case}\label{sect:proof_k}

In this section, we give the main arguments used to prove Proposition \ref{propo:Lambda} in the general case $k\in\N^*$. In particular, we want to show that the rate function we obtain does not depend on $k$. The proof of some important but technical results is postponed to Section \ref{sect:detailed_proof}.

Even if in Section \ref{sect:proof_1} above we have proved Proposition \ref{propo:Lambda} in the case $k=1$, we include this case in our general framework, and obtain an alternative proof.

To this aim, we make use of the strategy introduced in \cite{BLR} to study the properties of the ${\rm AMS}(n,k;a)$ algorithm. First, as explained in Section \ref{sect:prop_AMS}, we are allowed to restrict the study to the case where $X$ is exponentially distributed: it is enough to study the ${\rm AMS}_{\rm expo}(n,k;a_{\rm expo})$ algorithm, where $a_{\rm expo}=-\log(p)$.

Moreover, one of the main ideas is to consider the initial condition of the algorithm as an extra variable: for $x\in[0,a)$, we study the ${\rm AMS}_{\rm expo}(n,k;a_{\rm expo},x)$ algorithm. From now on, in this Section, and in Section \ref{sect:detailed_proof}, we only consider the exponential case and we omit the dependence.

\begin{defi}
We use the following notation: for any $(x,y)\in\R^2$
\begin{gather*}
f(y)=\exp(-y)\mathds{1}_{y\ge 0}\quad , \quad F(y)=\bigl(1-\exp(-y)\bigr)\mathds{1}_{y\ge 0}=\int_{-\infty}^{y}f(z)dz;\\
f(y;x)=\frac{f(y)}{1-F(x)}\mathds{1}_{y\ge x}\quad , \quad F(y;x)=\frac{F(y)-F(x)}{1-F(x)}\mathds{1}_{y\ge x}=\int_{-\infty}^{y}f(z;x)dz;\\
f_{n,k}(y;x)=k\binom{n}{k}F(y;x)^{k-1}f(y;x)\bigl(1-F(y;x)\bigr)^{n-k},\\
F_{n,k}(y;x)=\int_{x}^{y}f_{n,k}(z;x)dz.
\end{gather*}
Let $X$ be exponentially distributed with parameter $1$. Then $f$ (resp. $F$)) is the density (resp. the c.d.f.) of $\mathcal{L}(X)$. For $x\geq 0$, $f(\cdot;x)$ (resp. $F(\cdot;x)$) is the density (resp. the c.d.f.) of the conditional distribution $\mathcal{L}\bigl(X |X>x\bigr)$.

Finally, let $(X_1,\ldots,X_n)$ be i.i.d. with the distribution of $\mathcal{L}(X)$, with the associated order statistics $X_{(1)}<\ldots<X_{(n)}$. Then $f_{n,k}(\cdot;x)$ (resp. $F_{n,k}(\cdot;x)$) is the density (resp. the c.d.f.) of the $k$-th order statistic $X_{(k)}$.
\end{defi}




The main object we need to study is the following function $\Gamma_{n,k}$ of $\lambda\in \R$ (considered as a fixed parameter) and the initial condition $x\in[0,a]$
\begin{eqnarray}\label{eq:def_Gamma}
\Gamma_{n,k}(\lambda;x)&=\E \Bigl[\exp \bigl(n\lambda \log(\hat{p}^{n,k}(x)) \bigl) \Bigr]\\
&=\exp\bigl(\Lambda_{n,k}(n\lambda;x)\bigr). \nonumber
\end{eqnarray}
Notice that we include $x=a$ in the domain of definition of the functions $\Gamma_{n,k}$ and $\Lambda_{n,k}$ (defined by \eqref{eq:Lambda}). It is also important to remark that we evaluate the latter at $(n\lambda;x)$.

We state several fundamental results which together yield Proposition \ref{propo:Lambda} in the $x$-dependent case; to get \eqref{eq:Lambda_conv} it is then enough to take $x=0$.

First, Proposition \ref{propo:eq_funct} gives a functional equation satisfied by $\Gamma_{n,k}(\lambda;\cdot)$ on $[0,a]$, for any value of the parameters $1\leq k< n$ and $\lambda\in\R$.

We use the following auxiliary function:
\begin{equation}\label{eq:def_Theta}
\Theta_{n,k}(\lambda;x)=\sum_{\ell=0}^{k-1}\exp\bigl(n\lambda\log(1-\frac{\ell}{n})\bigr)\Bigl(F_{n,\ell}(a;x)-F_{n,\ell+1}(a;x) \Bigr),
\end{equation}
with the convention $F_{n,0}(y;x)=\mathds{1}_{y\geq x}$.

\begin{propo}\label{propo:eq_funct}
For any $n\in\N^*$, $k\in\left\{1,\ldots,n-1\right\}$, and $\lambda\in \R$, the function $\Gamma_{n,k}(\lambda;\cdot)$ is solution on the interval $[0,a]$ of the functional equation (with the unknown $\Gamma$):
\begin{equation}\label{eq:eq_funct}
\Gamma(x)=\int_{x}^{a}  \exp\Bigl(n\lambda\log(1-\frac{k}{n})\Bigr)   \Gamma(y)f_{n,k}(y;x)dy+\Theta_{n,k}(\lambda;x). 
\end{equation}
\end{propo}

Notice that for the moment, it is not cleat that $\Gamma_{n,k}$ is the unique solution of the functional equation \eqref{eq:eq_funct}. We will prove this property below.

For completeness, we include a proof of this result, even if follows the same lines as Proposition $4.2$ in \cite{BLR}.
\begin{proof}[Proof of Propositon \ref{propo:eq_funct}]
We decompose the expectation according to the value of the (random) number of iterations $J^{n,k}(x)$  in the algorithm starting from $x$:
\begin{align*}
\Gamma_{n,k}(\lambda;x)&=\E \Bigl[\exp \bigl(n\lambda \log(\hat{p}^{n,k}(x)) \bigl) \Bigr]\\
&=\E \Bigl[\exp \bigl(n\lambda \log(\hat{p}^{n,k}(x)) \bigl) \mathds{1}_{J^{n,k}(x)=0}\Bigr]+\E \Bigl[\exp \bigl(n\lambda \log(\hat{p}^{n,k}(x)) \bigl) \mathds{1}_{J^{n,k}(x)\ge 1}\Bigr].
\end{align*}
First, since $\left\{J^{n,k}(x)=0\right\}=\left\{ Z^1\ge a\right\} = \bigcup_{\ell=0}^{k-1}\left\{ X_{(\ell+1)}\ge a > X_{(\ell)}\right\}$, we have
\begin{align*}
\E \Bigl[\exp \bigl(n\lambda \log(\hat{p}^{n,k}(x)) \bigl) \mathds{1}_{J^{n,k}(x)=0}\Bigr]&=\E \Bigl[\exp \bigl(n\lambda \log(C^{n,k}(x)) \bigl) \mathds{1}_{J^{n,k}(x)=0}\Bigr]\\
&=\sum_{\ell=0}^{k-1}\exp\Bigl(n\lambda\log(1-\frac{\ell}{n})\Bigr)\Bigl(F_{n,\ell}(a;x)-F_{n,\ell+1}(a;x) \Bigr)\\
&=\Theta_{n,k}(\lambda;x).
\end{align*}
Second, we use $\left\{J^{n,k}(x)\ge 1\right\}=\left\{ Z^1\le a\right\}$ and condition with respect to $Z^1$:
\begin{align*}
\E \Bigl[\exp \bigl(n\lambda &\log(\hat{p}^{n,k}(x)) \bigl) \mathds{1}_{J^{n,k}(x)\ge 1}\Bigr]=\E \Bigl[ \E\bigl[ \exp \bigl(n\lambda \log(\hat{p}^{n,k}(x)) \bigl) \big| Z^1 \bigr] \mathds{1}_{Z^1< a}\Bigr]\\
&=\E \Bigl[ \E\bigl[ \exp \bigl(n\lambda \log((1-k/n)^{J^{n,k}(x)-1}C^{n,k}(x)) +n\lambda\log(1-k/n)  \bigl) \big| Z^1 \bigr] \mathds{1}_{Z^1< a}\Bigr]\\
&=\exp\Bigl(n\lambda\log(1-\frac{k}{n})\Bigr)\E \Bigl[ \E\bigl[ \exp \bigl(n\lambda \log((1-k/n)^{J^{n,k}(Z^1)}C^{n,k}(Z^1))  \bigl) \big| Z^1 \bigr] \mathds{1}_{Z^1< a}\Bigr]\\
&=\exp\Bigl(n\lambda\log(1-\frac{k}{n})\Bigr)\E \Bigl[ \Gamma_{n,k}(Z^1;x) \mathds{1}_{Z^1< a}\Bigr]\\
&=\int_{x}^{a}  \exp\Bigl(n\lambda\log(1-\frac{k}{n})\Bigr)   \Gamma_{n,k}(\lambda;y)f_{n,k}(y;x)dy.
\end{align*}
We have used a kind of Markov property for the algorithm: up to taking into account for one more iteration, the algorithm behaves the same starting from $x$ or from $Z^1\in (x,a]$.
\end{proof}

Notice that the functional equation \eqref{eq:eq_funct} involves a simple factor depending only on $\lambda$, $n$ and $k$ in the integral, and that on both the left and the right-hand sides the function $\Gamma$ is evaluated at the same value of the parameter $\lambda$. These observations are consequences of the choice to prove a Large Deviations Principle for $\log(\hat{p}^{n,k})$ (instead of $\hat{p}^{n,k}$) thanks to the G\"artner-Ellis Theorem, and to conclude with the use of the contraction principle; the same trick was used in \cite{BGT} to prove the Central Limit Theorem, thanks to the delta-method and the use of Levy Theorem. If one replaces $\log(\hat{p}^{n,k}(x))$ with $\hat{p}^{n,k}(x)$ in \eqref{eq:def_Gamma}, then one obtains a more complicated functional equation where the observations above do not hold, and which is not easily exploitable. In particular, one does not obtain a nice counterpart of the fundamental result, Proposition \ref{propo:EDO} below.

We now state in Proposition \ref{propo:EDO} that solutions $\Gamma$ of the functional equation \eqref{eq:eq_funct} are in fact solutions of a linear Ordinary Differential Equation (ODE) of order $k$, with constant coefficients.
\begin{propo}\label{propo:EDO}
For any $n\in\N^*$, $k\in\left\{1,\ldots,n-1\right\}$, and $\lambda\in \R$, let $\Gamma$ be a solution of the functional equation \eqref{eq:eq_funct}. Then it is solution of the following linear ODE of order $k$:
\begin{equation}\label{eq:ODE}
\frac{d^k}{dx^k}\Gamma_{n,k}(\lambda;x)=\exp\Bigl(n\lambda\log(1-\frac{k}{n})\Bigr)\mu^{n,k}\Gamma_{n,k}(\lambda;x)+\sum_{m=0}^{k-1}r_{m}^{n,k}\frac{d^m}{dx^m}\Gamma_{n,k}(\lambda;x).
\end{equation}
The coefficients $\mu^{n,k}$ and $(r_{m}^{n,k})_{0\leq m\leq k-1}$ satisfy the following properties:
\begin{equation}\label{eq:rec_coeffs}
\begin{gathered}
\mu^{n,k}=(-1)^{k}n\ldots (n-k+1)\\
\nu^k-\sum_{m=0}^{k-1}r_{m}^{n,k} \, \nu^m=(\nu-n)\ldots (\nu-n+k-1) \quad \text{ for all } \nu\in\R.
\end{gathered}
\end{equation}
\end{propo}

A sketch of proof of this result is postponed to Section \ref{sect:detailed_proof}. It uses the same arguments as to prove the corresponding functional equation in \cite{BLR}. For the proof of \eqref{eq:rec_coeffs} in particular, we refer to that article.

To conclude on uniqueness of the solution of \eqref{eq:eq_funct}, and then prove asymptotic expansions on $\Gamma_{n,k}$, we prove the following Lemma.

\begin{lemme}\label{lemme:derivs}
For any fixed $k\in\left\{1,\ldots,\right\}$ and any $\lambda\in\R$, we have for any $m\in\left\{0,\ldots,k-1\right\}$
\begin{eqnarray}\label{eq:init_cond}
\frac{d^{m}}{dx^{m}}\Gamma_{n,k}(\lambda;x) \Big|_{x=a}&=\frac{d^{m}}{dx^{m}}\Theta_{n,k}(\lambda;x) \Big|_{x=a}\\
&\underset{n \to \infty}\sim n^{m}\bigl(1-\exp(-\lambda)\bigr)^m. \nonumber
\end{eqnarray}
\end{lemme}

By Cauchy-Lipschitz theory, the linear ODE \eqref{eq:ODE} with the conditions \eqref{eq:init_cond} at $x=a$ admits a unique solution; therefore it is clear that $\Gamma_{n,k}$ is the unique solution of \eqref{eq:eq_funct}.

\begin{rem}
To prove the Central Limit Theorem in \cite{BGT}, we used a similar result although in a weaker form: we only needed to prove $\frac{d^{m}}{dx^{m}}\Theta_{n,k}(\lambda;x) \Big|_{x=a}=\text{O}(n^m)$. Here we require a more precise asymptotic result in order to prove that the coefficient $\gamma_{n,k}^{1}(\lambda)$ defined in Proposition \ref{propo:EDO_resol} below converges to $1$ (in fact, we only need that it is bounded from below by a positive constant).
\end{rem}

We finally explain how to obtain asymptotic knowledge on $\Gamma_{n,k}(\lambda;x)$ and $\Lambda_{n,k}(n\lambda,x)$ when $n\rightarrow +\infty$. First, the $k$ roots $\bigl(\nu_{n,k}^{\ell}(\lambda)\bigr)_{1\leq \ell\leq k}$ of the polynomial equation associated with the linear ODE \eqref{eq:ODE} are pairwise distinct for $n$ large enough (the other parameters $\lambda$ and $k$ being fixed), and more precisely they satisfy \eqref{eq:roots}. As a consequence, the solution $\Gamma_{n,k}$ can be written (see \eqref{eq:sol_ODE}) as a linear combination of exponential functions $x\mapsto \exp\Bigl({\nu_{n,k}^{\ell}(\lambda)\left(x-a\right)}\Bigr)$. Finally, using the asymptotic expression for the derivatives of order $0,\ldots,k-1$ at $x=a$, we obtain a linear system of equations, solve it using the Cramer's formulae and obtain the asymptotic expression \eqref{eq:coeffs}. The proof is postponed to Section \ref{sect:detailed_proof}.

\begin{propo}\label{propo:EDO_resol}
Let $k\in\left\{1,\ldots,\right\}$ and $\lambda\in\R$ be fixed. Then for $n$ large enough, we have for any $x\in[0,a]$
\begin{equation}\label{eq:sol_ODE}
\Gamma_{n,k}(\lambda,x) = \sum_{\ell=1}^{k} \gamma_{n,k}^{\ell}(\lambda) \exp\Bigl({\nu_{n,k}^{\ell}(\lambda)\left(x-a\right)}\Bigr),
\end{equation}
where
\begin{equation}\label{eq:roots}
\nu_{n,k}^{\ell}(\lambda)\sim n\Bigl(1-e^{-\lambda}e^{i2\pi \frac{(\ell-1)}{k}}\Bigr)
\end{equation}
and
\begin{equation}\label{eq:coeffs}
\gamma_{n,k}^{\ell}(\lambda)\rightarrow \mathds{1}_{\ell=1}.
\end{equation}
\end{propo}

We now conclude and prove Proposition \ref{propo:LDP_log}, namely the Large Deviations Principle for $\bigl(\mathcal{L}(\log(\hat{p}^{n,k}))\bigr)_{n> k}$.

We start with the case $k>1$. Then for any $\ell\in\left\{2,\ldots,k\right\}$ we have for any $\lambda\in\R$
$$\text{Re}\Bigl(1-e^{-\lambda}e^{i2\pi (\ell-1)/k}\Bigr)>\text{Re}\Bigl(1-e^{-\lambda}\Bigr).$$
As a consequence, for $x<a$ we have when $n\rightarrow +\infty$
$$e^{\nu_{n,k}^{\ell}(\lambda)\left(x-a\right)}={\rm o}\Bigl(e^{1-\exp(-\lambda))\left(x-a\right)} \Bigr),$$
and thus
$$
\frac{1}{n}\Lambda_{n,k}(n\lambda;x)=\frac{1}{n}\log(\Gamma_{n,k}(\lambda;x))\underset{n\rightarrow +\infty}\sim \nu_{n,k}^{1}(\lambda)(x-a)\underset{n\rightarrow +\infty}\rightarrow (1-e^{-\lambda})(x-a).
$$

When $k=1$, the linear ODE \eqref{eq:ODE} is of order $1$, and it is easy to check that
$$\Gamma_{n,1}(\lambda;x)=\exp\Bigl(\nu_{n,k}^{1}(\lambda)(x-a) \Bigr),$$
so that the same asymptotic result as above holds.

It remains to take $x=a$, and to recall that $a=-\log(p)$ if $p=\PP(X>a)$ and $X$ is exponentially distributed with parameter $1$.

This concludes the proof of Proposition \ref{propo:LDP_log}.

\section{Comparison with other algorithms}\label{sect:comp}

We propose a comparison (in terms of large deviations) of the Adaptive Multilevel Splitting algorithm with the two other methods described in the Introduction: a direct, naive Monte-Carlo method, based on a non-interacting system of replicas with the same size (see the estimator \eqref{eq:cMC_intro}) , and a non-adaptive version of multilevel splitting (see the estimator \eqref{eq:estim_MS_intro}).

In the first case, we obtain that large deviations are much less likely for the AMS algorithm than for the crude Monte-Carlo method. In the second case, we show that the AMS estimator is more efficient than the non-adaptive one taken in the limit of a large number $N$ of fixed levels.

These results are consistent with the cost analysis and the comparison based on the central limit theorem, see \cite{BLR}, \cite{BGT}, \cite{CerouDel-MoralFuronGuyader2012}, \cite{CerouGuyader2014}.

\subsection{Crude Monte-Carlo}

We compare the performance of the AMS algorithm with the use of a Crude Monte-Carlo estimation in the large $n$ limit.

Let $(X_m)_{m\in\N^*}$ a sequence of independent and identically distributed random variables, each one being equal in law with $X$.

Then for any $n\in\N^*$
\begin{equation}
\overline{p}_n=\frac{1}{n}\sum_{m=1}^{n}\mathds{1}_{X_m>a}
\end{equation}
is an unbiased estimator of $p$.

It is a classical result (Theorem $2.2.3$ in \cite{DZ}) due to Cramer that the sequence $\bigl(\mathcal{L}(\overline{p}_n)\bigr)_{n\in\N^*}$ satisfies a Large Deviations Principle with the rate function (case of Bernoulli random variables, see Exercice $2.2.23$ in \cite{DZ}):
\begin{equation}\label{eq:def_Ical}
\mathcal{I}(y)=\begin{cases} +\infty \text{ if } y\notin(0,1) \\ y\log\left(\frac{y}{p}\right)+(1-y)\log\left(\frac{1-y}{1-p}\right) \text{ if } y\in(0,1).\end{cases}
\end{equation}

The comparison between the algorithms is based on the following result:
\begin{propo}
For any $p\in(0,1)$ and any $y\in(0,1)$, we have
\begin{gather*}
I(y) \ge \mathcal{I}(y),\\
I(y)=\mathcal{I}(y) \quad \text{ if and only if } \quad y=p.
\end{gather*}
\end{propo}

\begin{proof}
We explicitly mention the dependence of $I$ and of $\mathcal{I}$ with respect to $p$, and we define
$$D(y,p)=I(y,p)-\mathcal{I}(y,p).$$
It is clear that $D(p,p)=0$ for any $p\in (0,1)$.
We compute that
$$\frac{\partial D(y,p)}{\partial p}=\frac{1-y}{p\log(p)}\Bigl(\frac{\log(y)}{1-y}-\frac{\log(p)}{1-p}\Bigr);$$
since the function $t\mapsto \frac{\log(t)}{1-t}$ is strictly decreasing on $(0,1)$ (as can be seen by computing its first and second order derivatives), we see that for any $y,p\in(0,1)^2$ we have the inequalities
$$\frac{\partial D(y,p)}{\partial p}>0 \quad \text{if} \quad y>p \quad \text{ and } \quad
\frac{\partial D(y,p)}{\partial p}<0 \quad \text{if} \quad y<p.$$
Using $D(p,p)=0$, it is easy to conclude.
\end{proof}

Now let $\epsilon\in(0,\max(p,1-p))$; then for $n$ large we have
$$
\frac{\PP(\hat{p}^{n,k}-p>\epsilon)}{\PP(\overline{p}_n-p>\epsilon)}=\exp\Bigl(n\Delta(\epsilon,n) \Bigr)\rightarrow 0,
$$
exponentially fast, since we have by the Large Deviations Principles $\Delta(\epsilon,n)\rightarrow \mathcal{I}(p+\epsilon)-I(p+\epsilon)<0$ when $n\rightarrow +\infty$ (notice that both $\mathcal{I}$ and $I$ are increasing on $(p,1)$).

The same arguments apply to get
$$\frac{\PP(\hat{p}^{n,k}-p<-\epsilon)}{\PP(\overline{p}_n-p<-\epsilon)}\rightarrow 0.$$
As a consequence, the probability of observing large deviations from the mean $p$ is much smaller for the AMS algorithm than when using a crude Monte-Carlo estimator, in the large $n$ limit. This statement is a new way of expressing the efficiency of the AMS algorithm.

Notice that in the discussion above we have not assumed that we are estimating a probability in a rare event regime: the conclusion holds for any $p\in(0,1)$. Now it is also instructive to compare $I((1+\epsilon)p)$ and $\mathcal{I}((1+\epsilon)p)$ for a given $\epsilon\in(0,1)$ and when $p\rightarrow 0$: it amounts at looking at deviations of the relative error, and we have
\begin{gather*}
\lim_{n\rightarrow +\infty}\frac{1}{n}\log\Bigl(\PP(\frac{\hat{p}^{n,k}-p}{p}>\epsilon)\Bigr)=-I\bigl(p(1+\epsilon)\bigr)\sim_{p\rightarrow 0} -\frac{\bigl(\log(1+\epsilon)\bigr)^2}{-2\log(p)}\\
\lim_{n\rightarrow +\infty}\frac{1}{n}\log\Bigl(\PP(\frac{\overline{p}_n-p}{p}>\epsilon)\Bigr)=-\mathcal{I}\bigl(p(1+\epsilon)\bigr)\sim_{p\rightarrow 0} -p\bigl((1+\epsilon)\log(1+\epsilon)-\epsilon\bigr).
\end{gather*}
Given $\delta>0$, in order to have a probability lower than $\delta$ that the relative error is larger than $\epsilon$, in the small $p$ limit, one thus needs a number of replicas $n$ which scales like $1/p$ when using the crude Monte-Carlo method, while it scales like $-\log(p)$ (which is much smaller) when using the AMS algorithm. Moreover, since the expected workload is of size $n$ when using the Monte-Carlo method and of size $-n\log(p)$ when using the AMS algorithm, it is clear that in terms of large deviations from the mean the AMS algorithm is more efficient than the crude Monte-Carlo method.

Notice that this discussion is consistent with the conclusions coming from the Central Limit Theorem, where in the regime $p\rightarrow 0$ the asymptotic variance is equivalent to $p$ when using the crude Monte-Carlo method and $-p^2\log(p)$ when using the AMS algorithm: to obtain reliable confidence intervals on the relative error, the number of replicas $n$ scales in the same way.


\subsection{Non-adaptive Multilevel Splitting}

We now compare the rate function $I$ obtained for the Large Deviations Principle on the AMS algorithm, with the one we obtain when using a deterministic (non-adaptive) sequence of levels.

Namely, using Assumption \ref{ass:cdf:c0}, we decompose the probability as a telescoping product of $N\in\N^*$ conditional probabilities
\begin{equation}\label{eq:decompMS}
p=\PP(X>a)=\prod_{i=1}^{N}\PP(X>a_i \big| X>a_{i-1}),
\end{equation}
associated with a given non-decreasing sequence of levels $a_0=0<a_1<\ldots<a_N=a$. We denote by $p^{(i)}=\PP(X>a_i \big| X>a_{i-1})$ the $i$-th conditional probability. The sequence is of size $N$ and we study the asymptotic regime $N\rightarrow +\infty$.

We can define an unbiased estimator of $p$ as follows: let $n\in\N^*$ and set
\begin{equation}
\hat{p}_n^N=\prod_{i=1}^{N}\overline{p}_{n}^{(i)},
\end{equation}
where $\bigl(\overline{p}_{n}^{(i)}\bigr)_{1\leq i\leq N}$ is a family of independent random variables, where each $\overline{p}_{n}^{(i)}$ is a Crude Monte-Carlo estimator (as defined in the section above) for the probability $p^{(i)}$ with $n$ realizations. More precisely, let $\bigl(X_{m}^{(i)}\bigr)_{1\leq m\leq n, 1\leq i\leq N}$ be independent random variables, such that $\mathcal{L}\bigl(X_{m}^{(i)}\bigr)=\mathcal{L}(X | X>a_{i-1})$, and set
\begin{equation}
\overline{p}_{n}^{(i)}=\frac{1}{n}\sum_{m=1}^{n}\mathds{1}_{X_{m}^{(i)}>a_{i}}.
\end{equation}
From a practical point of view, notice that the computation of these estimators requires the sampling of random variables according to the conditional distribution $\mathcal{L}(X | X>a_{i-1})$ for each $i\in\left\{1,\ldots,N\right\}$, just like for the adaptive version.




Here $n$ thus denotes the number of replicas used for the estimation of the probabilities in both the adaptive and the non-adaptive versions. We needed the extra parameter $N$ to denote the number of iterations (\textit{i.e.} the length of the sequence of levels) of the algorithm, while we know that the average number of iterations is of the order $-\frac{n\log(p)}{k}$ in the adaptive case. Therefore, to study the non-adaptive version, we first let $n\rightarrow +\infty$, and then analyze the behavior of the asymptotic quantities with respect to $N$ (in the limit $N\rightarrow +\infty$), while for the adaptive version we need to pass to the limit only once, namely $n\rightarrow +\infty$.

Clearly, by the independence properties of the random variables introduced here we have
$$\E[\hat{p}_{n}^{N}]=p.$$
Moreover, it is well-known that, for a given value of $N$ (the length of the sequence of levels) the asymptotic variance (when $n$ goes to $+\infty$) is minimized when $p^{(i)}=p^{1/N}$ for any $i\in\left\{1,\ldots,N\right\}$ (\text{i.e.} the conditional probabilities in \eqref{eq:decompMS} are equal); moreover the asymptotic variance is a decreasing function of $N$, which converges to $\frac{-p^2\log(p)}{n}$ when $N\rightarrow +\infty$. From a practical point of view, the computation of the associated sequence of levels $a_1,\ldots, a_{N-1}$ is \textit{a priori} difficult: the adaptive version overcomes this issue, and in the regime $N\rightarrow +\infty$ both the non-adaptive and the adaptive version have the same statistical properties.

As a consequence, from now on we assume that $p^{(i)}=p^{1/N}$ for any $i\in\left\{1,\ldots,N\right\}$.

For any $i\in\left\{1,\ldots,N\right\}$, $\bigl(\mathcal{L}(\overline{p}_{n}^{(i)})\bigr)_{n\in\N^*}$ satisfies a Large Deviations Principle with the rate function (see \eqref{eq:def_Ical})
\begin{equation}
\mathcal{I}_N(y)=\begin{cases} +\infty \text{ if } y\notin(0,1) \\ y\log\left(\frac{y}{p^{1/N}}\right)+(1-y)\log\left(\frac{1-y}{1-p^{1/N}}\right) \text{ if } y\in(0,1).\end{cases}
\end{equation}
Since for any $n\in \N^*$ the random variables $\bigl(\overline{p}_{n}^{(i)}\bigr)_{1\leq i\leq N}$ are independent, it is easy to generalize this statement as follows. The sequence $\bigl(\mathcal{L}(\overline{p}_{n}^{(1)},\ldots,\overline{p}_{n}^{(N)})\bigr)_{n\in \N^*}$ satisfies a Large Deviations Principle in $\R^N$ with the rate function (with abuse of notation $\mathcal{I}_N$ refers both to the function depending on a $1$-dimensional or a $N$-dimensional variable)
\begin{equation}
\mathcal{I}_N(y_1,\ldots,y_N)=\sum_{i=1}^{N}\mathcal{I}_N(y_i).
\end{equation}
Now as a consequence of the contraction principle, since $\hat{p}_{n}^{N}=\prod_{i=1}^{N}\overline{p}_{n}^{(i)}$, the sequence $\bigl(\hat{p}_{n}^{N}\bigr)_{n\in\N^*}$ also satisfies a Large Deviations Principle with the rate function
\begin{equation}
I_N(y)=\inf\left\{ \mathcal{I}_N(y_N,\ldots,y_N) ~;~ y=\prod_{i=1}^{N}y_i\right\}.
\end{equation}
On the one hand, it is clear that if $y\notin (0,1)$, then $I_N(y)=+\infty$. Indeed, for any $(y_1,\ldots,y_N)$ satisfying the constraint $y=\prod_{i=1}^{N}y_i\notin(0,1)$, at least one of the $y_i$'s satisfies $y_i\notin (0,1)$, which yields $\mathcal{I}_{N}(y_i)=\mathcal{I}_N(y_1,\ldots,y_n)=+\infty$.

On the other hand, by definition of $I_N$, we have for any $y\in (0,1)$
\begin{align*}
I_N(y)&\le \mathcal{I}_N(y^{1/N},\ldots,y^{1/N})=N\mathcal{I}_N(y^{1/N})\\
&=Ny^{1/N}\log\bigl(\frac{y^{1/N}}{p^{1/N}}\bigr)+N(1-y^{1/N})\log\bigl(\frac{1-y^{1/N}}{1-p^{1/N}}\bigr)\\
&\underset{N \to \infty}\rightarrow \log(y)-\log(p)-\log(y)\log\bigl(\frac{\log(y)}{\log(p)}\bigr)=I(y).
\end{align*}

For our purpose, this inequality is sufficient.

We now interpret the previous inequality in terms of asymptotic estimates for deviations of $\hat{p}_{n}^{N}$ and of $\hat{p}^{n,k}$ with respect to their expected value $p$. Let $\epsilon>0$, then we have by definition of the Large Deviations Principle with rate function $I_N$
\begin{align*}
\liminf_{n\rightarrow +\infty}\frac{1}{n}\log\Bigl(\PP\bigl(\big|\hat{p}_{n}^{N}-p\big|>\epsilon \bigr) \Bigr)&\geq -\inf\left\{I_N(y) ~; |y-p|\geq \epsilon\right\}\\
&\geq -\inf\left\{N\mathcal{I}_N(y^{1/N}) ~;~ |y-p|\geq \epsilon\right\}\\
&\geq -\min\left\{N\mathcal{I}_N((p+\epsilon)^{1/N}),N\mathcal{I}_N((p-\epsilon)^{1/N})\right\},
\end{align*}
using that $\mathcal{I}_N$ is non-increasing on $(-\infty,p^{1/N})$ and non-decreasing on $(p^{1/N},+\infty)$.

To conclude, notice that
\begin{align*}
\lim_{N\rightarrow +\infty}-\min\left\{N\mathcal{I}_N((p+\epsilon)^{1/N}),N\mathcal{I}_N((p-\epsilon)^{1/N})\right\}&=-\min\left\{I(p+\epsilon),I(p-\epsilon)\right\}\\
&=\lim_{n\rightarrow +\infty}\frac{1}{n}\log\Bigl(\PP\bigl(\big|\hat{p}^{n,k}-p\big|>\epsilon \bigr) \Bigr).
\end{align*}

We can thus assess that the Adaptive Multilevel Splitting algorithm is more efficient (in a large sense) than the non-adaptive version in terms of large deviations when the number of replicas $n$ goes to $+\infty$ and in the limit of large number $N$ if levels.

\section{Proof of the technical estimates}\label{sect:detailed_proof}

In this section, we give detailed proofs for the technical auxiliary results used in Section \ref{sect:proof_k}.

\begin{proof}[Proof of Proposition \ref{propo:EDO}]
We proceed by recursion, like in the proof of Proposition $6.4$ in \cite{BLR} and Lemma $2$ in \cite{BGT}. We fix the values of $1\leq k<n$ and of $\lambda\in\R$.

Differentiating recursively with respect to $x$, for any $0\leq l\leq k-1$ and for any $0\leq x\leq a$ we have (for a family of coefficients described by \eqref{eq:recursion} below)
\begin{eqnarray}\label{eq:recur_p}
\frac{d^l}{dx^l}\left(\Gamma_{n,k}(\lambda;x)-\Theta_{n,k}(\lambda;x)\right)&=&\mu_{l}^{n,k}\exp\Bigl(n\lambda\log(1-\frac{k}{n})\Bigr)\int_{x}^{a}\Gamma_{n,k}(\lambda;y)f_{n,k-l}(y;x) dy\nonumber\\
&&+\sum_{m=0}^{l-1}r_{m,l}^{n,k}\frac{d^m}{dx^m}\left(\Gamma_{n,k}(\lambda;x)-\Theta_{n,k}(\lambda;x)\right),
\end{eqnarray}
and that differentiating once more we get
\begin{eqnarray}\label{eq:recur_p_fin}
\frac{d^k}{dx^k}\left(\Gamma_{n,k}(\lambda;x)-\Theta_{n,k}(\lambda;x)\right)&=&\mu^{n,k}\exp\Bigl(n\lambda\log(1-\frac{k}{n})\Bigr)\Gamma_{n,k}(\lambda;x)\nonumber\\
&&+\sum_{m=0}^{k-1}r_{m}^{n,k}\frac{d^m}{dx^m}\left(\Gamma_{n,k}(\lambda;x)-\Theta_{n,k}(\lambda;x)\right),
\end{eqnarray}
with $\mu^{n,k}:=\mu_{k}^{n,k}$ and $r_{m}^{n,k}:=r_{m,k}^{n,k}$. 

The coefficients satisfy
\begin{equation}\label{eq:recursion}
\begin{gathered}
\mu_{0}^{n,k}=1,\mu_{l+1}^{n,k}=-(n-k+l+1)\mu_{l}^{n,k};\\
\begin{cases}
r_{0,l+1}^{n,k}=-(n-k+l+1)r_{0,l}^{n,k}, \quad \text{if } l>0,\\
r_{m,l+1}^{n,k}=r_{m-1,l}^{n,k}-(n-k+l+1)r_{m,l}^{n,k}, \quad 1\leq m\leq l,\\
r_{l,l}^{n,k}=-1.
\end{cases}
\end{gathered}
\end{equation}
Notice that these coefficients do not depend on $\lambda$, and are the same as in \cite{BLR} and \cite{BGT}. Properties \eqref{eq:rec_coeffs} are proved in \cite{BLR}.

Thanks to \eqref{eq:rec_coeffs}, for all $j\in\left\{0,\ldots,k-1\right\}$ and any $x\in[0,a]$ we have
$$\frac{d^k}{dx^k}\exp\left((n-k+j+1)(x-a)\right)=\sum_{m=0}^{k-1}r_{m}^{n,k}\frac{d^m}{dx^m}\exp\left((n-k+j+1)(x-a)\right).$$
Using the expression of $F_{n,k}$, straightforward computations show that $\Theta_{n,k}(\lambda;\cdot)$ is a linear combination of the exponential functions $z\mapsto \exp(-nz),\ldots,\exp(-(n-k+1)z)$; therefore
$$\frac{d^k}{dx^k}\Theta_{n,k}(t,x)=\sum_{m=0}^{k-1}r_{m}^{n,k}\frac{d^m}{dx^m}\Theta_{n,k}(t,x),$$
and thus \eqref{eq:recur_p_fin} gives \eqref{eq:ODE}.

\end{proof}

\begin{proof}[Proof of Lemma \ref{lemme:derivs}]

From \eqref{eq:recur_p}, the equality in \eqref{eq:init_cond} is clear.

We claim that for any $0\leq m\leq k-1$ and any $0\leq \ell \leq k-1$
\begin{equation}\label{eq:proof_derivs}
\frac{d^{m}}{dx^{m}}\Bigl(F_{n,\ell}(a;x)-F_{n,\ell+1}(a;x) \Bigr)\big|_{x=a}\underset{n \to \infty}\sim n^{m}\binom{m}{\ell}(-1)^{\ell}.
\end{equation}
In particular, $\frac{d^{m}}{dx^{m}}\Bigl(F_{n,\ell}(a;x)-F_{n,\ell+1}(a;x) \Bigr)\big|_{x=a}=0=\binom{m}{\ell}$ for $n$ large enough as soon as $\ell>m$.
Conclusion is then straightforward: using the definition \eqref{eq:def_Theta} of $\Theta_{n,k}$, we get
\begin{align*}
\frac{1}{n^m}\frac{d^{m}}{dx^{m}}\Theta_{n,k}(\lambda;x)\big|_{x=a}
&=\frac{1}{n^m}\sum_{\ell=0}^{k-1}\frac{d^{m}}{dx^{m}}\exp\bigl(n\lambda\log(1-\frac{\ell}{n})\bigr)\Bigl(F_{n,\ell}(a;x)-F_{n,\ell+1}(a;x) \Bigr)\big|_{x=a}\\
&\underset{n \to \infty}\rightarrow \sum_{\ell=0}^{m}\binom{m}{\ell}(-1)^{\ell}\exp\bigl(-\ell\lambda\bigr)\\
&=\Bigl(1-\exp\bigl(-\lambda\bigr)\Bigr)^m.
\end{align*}

We now prove \eqref{eq:proof_derivs} by induction on $m$.

We first consider $m=0$. Then for any $\ell\in \N^*$ we have $F_{n,\ell}(a;a)=0$ and $F_{n,0}(a;a)=1$ (by the convention $F_{n,0}(y;x)=\mathds{1}_{y\geq x}$), and \eqref{eq:proof_derivs} holds.

Let us also consider $m=1$, when $k\geq 2$. Then $\frac{d}{dx}F_{n,0}(a;x)\big|_{x=a}=0$, while for any $x\leq a$
$$\frac{d}{dx}F_{n,\ell}(a;x)=\frac{d}{dx}F_{n,\ell}(a-x;0)=-f_{n,\ell}(a-x;0)=-f_{n,\ell}(a;x)$$
as a consequence of the absence of memory property of the exponential distribution. 

Now since $f_{n,\ell}(a,a)=n\mathds{1}_{\ell=1}$, we get \eqref{eq:proof_derivs} for $m=1$.

The induction is based on the following relations (deduced from elementary computations; for a proof see \cite{BLR}, Section $6.3$)
\begin{equation}\label{formula:d/dxk}
\left\{
\begin{array}{l}
\begin{gathered}
\frac{d}{dx}f_{n,1}(y;x)=nf_{n,1}(y;x).\\
\text{ for $\ell \in \{2, \ldots , n-1\}$}, \,\frac{d}{dx}f_{n,\ell}(y;x)=(n-\ell+1)\bigl(f_{n,\ell}(y;x)-f_{n,\ell-1}(y;x)\bigr).
\end{gathered}
\end{array}
\right.
\end{equation}
Thanks to the first formula in \eqref{formula:d/dxk}, we easily get \eqref{eq:proof_derivs} for $\ell=0$ by induction on $m$.

If now $\ell\in\left\{1,\ldots,k-1\right\}$, we have the recursive formula for $m\geq 1$
\begin{eqnarray*}
\frac{d^{m+1}}{dx^{m+1}}\Bigl(F_{n,\ell}(a;x)-F_{n,\ell+1}(a;x)\Bigr)\big|_{x=a}&=&\frac{d^{m}}{dx^{m}}\Bigl(f_{n,\ell+1}(a;x)-f_{n,\ell}(a;x)\Bigr)\big|_{x=a}\\
&=&~(n-\ell)\frac{d^{m-1}}{dx^{m-1}}\Bigl(f_{n,\ell+1}(a;x)-f_{n,\ell}(a;x)\Bigr)\big|_{x=a}\\
&&-(n-\ell+1)\frac{d^{m-1}}{dx^{m-1}}\Bigl(f_{n,\ell}(a;x)-f_{n,\ell-1}(a;x)\Bigr)\big|_{x=a}\\
&=&~(n-\ell)\frac{d^{m}}{dx^{m}}\Bigl(F_{n,\ell}(a;x)-F_{n,\ell+1}(a;x)\Bigr)\big|_{x=a}\\
&&-(n-\ell+1)\frac{d^{m}}{dx^{m}}\Bigl(F_{n,\ell-1}(a;x)-F_{n,\ell}(a;x)\Bigr)\big|_{x=a}
\end{eqnarray*}
Finally using the induction hypothesis and obtain
\begin{align*}
\frac{1}{n^{m+1}}\frac{d^{m+1}}{dx^{m+1}}\Bigl(F_{n,\ell}(a;x)-F_{n,\ell+1}(a;x)\Bigr)\big|_{x=a}&\underset{n\rightarrow +\infty}\rightarrow (-1)^{\ell}\binom{m}{\ell}-(-1)^{\ell-1}\binom{m}{\ell-1}\\
&=(-1)^{\ell}\binom{m+1}{\ell}.
\end{align*}

This concludes the proof of Lemma \ref{lemme:derivs}.

%

%
%
%
%

\end{proof}

\begin{proof}[Proof of Proposition \ref{propo:EDO_resol}]
The $\nu_{n,k}^{\ell}(\lambda)$ are the roots of the caracteristic equation associated with the linear ODE \eqref{eq:ODE}:
$$\frac{(n-\nu)...(n-k+1-\nu)}{n...(n-k+1)}-\exp\Bigl(n\lambda\log(1-\frac{k}{n})\Bigr)=0,$$
which can be rewritten as a polynomial equation of degree $k$ with respect to the variable $\overline{\nu}_n=\frac{\nu}{n}$:
$$\frac{(1-\overline{\nu}_n)...(1-\frac{k-1}{n}-\overline{\nu}_n)}{1...(1-\frac{k-1}{n})}-\exp\Bigl(n\lambda\log(1-\frac{k}{n})\Bigr)=0,$$
where $\exp\Bigl(n\lambda\log(1-\frac{k}{n})\Bigr)\underset{n\rightarrow +\infty} \rightarrow \exp(-k\lambda)$.

By continuity of the roots of polynomials of degree $k$ with respect to the coefficients, we get that for all $\ell\in\left\{1,\ldots,k\right\}$ (with an appropriate ordering of the roots)
$$\frac{\nu_{n,k}^{\ell}(\lambda)}{n}\rightarrow \overline{\nu}_{\infty,k}^{\ell}(\lambda)$$
where $(1-\overline{\nu}_{\infty,k}^{\ell}(\lambda))^k=e^{-k\lambda}$. This identity immediately yields \eqref{eq:roots}.

As a consequence, for $n$ large enough the roots $\nu_{n,k}^{\ell}(\lambda)$ are pairwise distinct. Then \eqref{eq:sol_ODE} holds for some complex numbers $\gamma_{n,k}^{\ell}(\lambda)$, where $\ell\in\left\{1,\ldots,k\right\}$. Thanks to \eqref{eq:sol_ODE} evaluated at $x=a$, these coefficients are solution of the following linear system of equations:
\begin{equation}
\left\{
\begin{array}{l}
\gamma_{n,k}^{1}(\lambda)+ ... + \gamma_{n,k}^{k}(\lambda)=\Gamma_{n,k}(\lambda;x) \big|_{x=a},\\
\gamma_{n,k}^{1}(\lambda)\nu_{n,k}^{1}(\lambda) + ... + \gamma_{n,k}^{k}(\lambda)\nu_{n,k}^{k}(\lambda) =\frac{d}{dx}\Gamma_{n,k}(\lambda;x) \big|_{x=a},\\
\vdots \\
\gamma_{n,k}^{1}(\lambda)\left(\nu_{n,k}^{1}(\lambda)\right)^{k-1} + ... + \gamma_{n,k}^{k}(\lambda)\left(\nu_{n,k}^{k}(\lambda)\right)^{k-1} =\frac{d^{k-1}}{dx^{k-1}}\Gamma_{n,k}(\lambda;x) \big|_{x=a}.
\end{array}
\right.
\end{equation}
This system is equivalent with
\begin{equation}
\left\{
\begin{array}{l}
\gamma_{n,k}^{1}(\lambda)+ ... + \gamma_{n,k}^{k}(\lambda)=\Gamma_{n,k}(\lambda;x) \big|_{x=a}\underset{n\rightarrow +\infty}\rightarrow 1,\\
\gamma_{n,k}^{1}(\lambda)\overline{\nu}_{n,k}^{1}(\lambda) + ... + \gamma_{n,k}^{k}(\lambda)\overline{\nu}_{n,k}^{k}(\lambda) =\frac{1}{n}\frac{d}{dx}\Gamma_{n,k}(\lambda;x) \big|_{x=a}\underset{n\rightarrow +\infty}\rightarrow \overline{\nu}_{\infty,k}^{1}(\lambda),\\
\vdots \\
\gamma_{n,k}^{1}(\lambda)\overline{\nu}_{n,k}^{1}(\lambda)^{k-1} + ... + \gamma_{n,k}^{k}(\lambda)\overline{\nu}_{n,k}^{k}(\lambda)^{k-1} =\frac{1}{n^{k-1}}\frac{d^{k-1}}{dx^{k-1}}\Gamma_{n,k}(\lambda;x) \big|_{x=a}\underset{n\rightarrow +\infty}\rightarrow\overline{\nu}_{\infty,k}^{1}(\lambda)^{k-1},
\end{array}
\right.
\end{equation}
thanks to \eqref{eq:init_cond} and \eqref{eq:roots}, where $\overline{\nu}_{n,k}^{\ell}(\lambda)=\frac{\nu_{n,k}^{\ell}(\lambda)}{n}\underset{n\rightarrow +\infty}\rightarrow \overline{\nu}_{\infty,k}^{\ell}(\lambda)$.

It is now easy to get \eqref{eq:coeffs}, which concludes the proof of Proposition \ref{propo:EDO_resol}.

\end{proof}

\section{Conclusion and perspectives}\label{sect:conclusion}

We have established (Theorem \ref{th:LDP}) a Large Deviations Principle result for the Adaptive Multilevel Splitting ${\rm AMS}(n,k)$ Algorithm in an idealized setting, when the number of replicas $n$ goes to infinity while the parameter $k$ and the threshold $a$ remain fixed. The rate function does not depend on $k$: when $k=1$, the proof is very simple and uses an interpretation of the algorithm with a Poisson process (the number of iterations follows a Poisson distribution). When $k>1$, we rely on a functional equation technique which was already used to prove unbiasedness and asymptotic normality of the estimator in the previous works \cite{BLR} and \cite{BGT}.

We were able to relate the efficiency of the algorithm with this Large Deviations result, with a comparison with two algorithms (see Section~\ref{sect:comp}): a crude Monte-Carlo method and a non-adaptive version. More generally, in other situations Large Deviations could be a powerful tool to compare adaptive or non-adaptive multilevel splitting algorithms, instead of resorting only on comparison of asymptotic variances associated with central limit theorems.

Let us mention a few open directions for future works. First, it should be interesting to look at the regime where $k$ also goes to infinity, with $k/n$ converging to a proportion $\alpha\in(0,1)$. We expect to prove that the optimal rate function is obtained for $\alpha$ decreasing to $0$: indeed, the asymptotic variance is minimized in this regime. A comparison with a non-adaptive version of the algorithm is expected to show that the adaptive algorithm behaves (in terms of large deviations) like the non-adaptive version when the number of replicas and of levels goes to infinity, like in the regime we have studied in this paper.

A severe restriction is given by the so-called idealized setting: we need to know how to sample according to the conditional distribution $\mathcal{L}(X |X>x)$. In practice, and especially when computing crossing probabilities for high dimensional metastable stochastic processes, it is not satisfied and the multilevel splitting algorithm needs to use an importance function to define appropriate levels, and at each step the computation of the new sample uses the one at the previous iteration (thanks to a branching procedure of the successful trajectories). A natural question is whether one can prove a Large Deviations Principle in such a framework, and study quantitatively how the rate function depends on the importance function.

In fact, when using both non-adaptive (see \cite{GarvelsKroeseVan-Ommeren2002}, \cite{GlassermanHeidelbergerShahabuddinZajic1998}) and adaptive (\cite{BGGLR}, in preparation) multilevel splitting algorithms, one may observe a very large difference between the value of the estimator (averaged over a number $M$ of independent realizations) and the true result, or between the results obtained for different choices of the importance function. Even if the estimator of the probability is unbiased, in such situations one observes an \emph{apparent bias} toward smaller values if $M$ is not sufficiently large. This phenomenon is explained by specificity of the models: there are several channels to reach the region $B$ from $A$ (in the case of the estimation of crossing probabilities between metastable states of a Markov process), which may be sampled very differently when the importance function changes. It should be interesting to investigate the relation between this phenomenon and the Large Deviations Principle for the associated estimator.


{\small
\section*{\small Acknowledgments}

The author would like to thank B.~Bercu and A.~Richou for suggesting this work, and F.~C\'erou, A.~Guyader and M.~Rousset for helpful discussions and advice.
}

\bibliographystyle{alpha}

\newcommand{\etalchar}[1]{$^{#1}$}
\def\polhk#1{\setbox0=\hbox{#1}{\ooalign{\hidewidth
  \lower1.5ex\hbox{`}\hidewidth\crcr\unhbox0}}} \def\cprime{$'$}



\end{document}